\documentclass[11pt]{article}
\usepackage{indentfirst}
\usepackage{cite}

\usepackage{fullpage}

\usepackage{amssymb}
\usepackage{mathtools}

\usepackage{amsthm}
\newtheorem{theorem}{Theorem}[section]
\newtheorem{lemma}[theorem]{Lemma}
\newtheorem{corollary}[theorem]{Corollary}

\theoremstyle{definition}

\newtheorem{example}[theorem]{Example}

\theoremstyle{remark}
\newtheorem{remark}[theorem]{Remark}

\usepackage{verbatim}

\usepackage{moreverb}

\usepackage{multicol}

\usepackage{color}

\usepackage{graphicx}

\usepackage[pdfborder={0 0 0}]{hyperref}

\newcommand{\N}{\mathbb{N}}

\newcommand{\F}{\mathbb{F}}

\newcommand{\pp}{\mathfrak{p}}
\newcommand{\Aut}{\text{Aut}}
\newcommand{\Gal}{\text{Gal}}

\newcommand{\End}{\text{End}}

\title{Explicit Weil-pairing for Drinfeld Modules}
\author{Jeff Katen}

\begin{document}
\maketitle
\begin{abstract}
	The goal of this article is to define an analogue of the Weil-pairing for Drinfeld modules using explicit formulas and to deduce its main properties from these formulas. Our result generalizes the formula given for rank 2 Drinfeld modules by van der Heiden in~\cite{vanderheiden} and works as a more explicit, elementary proof of the Weil-pairing's existence than the proof appearing in~\cite{vanderheiden}.
\end{abstract}
\section{Introduction} Let $q$ be a prime power and let $\F_q$ be the finite field with $q$ elements. Let $\overline{\F}_q$ be the algebraic closure of $\F_q$. Let $A=\F_q[T]$ be the polynomial ring over $\F_q$ with indeterminate $T$. Let $K$ be a field containing $\F_q$ as a subfield and equipped with an $\F_q$-algebra homomorphism $\gamma:A\rightarrow K$. Let $\tau$ be the Frobenius endomorphism of $\overline{K}$ with respect to $\F_q$ defined by $\alpha\mapsto \alpha^q$, and denote by $K\{\tau\}$ the noncommutative ring of polynomials in $\tau$ with coefficients in $K$ and the commutation rule $\tau \alpha = \alpha^q \tau$ for any $\alpha\in K$. \par
Let $r\in \N$. A Drinfeld module of rank $r$ is a ring homomorphism $\phi:A\rightarrow K\{\tau\}$, 
\begin{align*}
a \mapsto \phi_a = \gamma(a) + \sum_{i=1}^{r\cdot \deg(a)} g_i(a) \tau^i, \qquad g_{r\cdot \deg(a)}(a)\neq 0.
\end{align*}
Note that $\phi$ is a homomorphism of $\F_q$-algebras. Since $K\{\tau\}\subset \End_{\F_q}(\overline{K})$, $\overline{K}$ inherits an $A$-module structure given for $a\in A$ and $\beta\in \overline{K}$ by $a\cdot \beta = \phi_a(\beta)$. For $a\in A$, denote by $\phi[a]$ the $a$-torsion of $\phi$, i.e. $\phi[a]=\ker(\phi_a)\subset \overline{K}$. Note that $\phi[a]$ is an $A$-submodule of $\overline{K}$, because for $\beta\in \phi[a]$ and $b\in A$, we have $\phi_a(\phi_b(\beta))=\phi_b(\phi_a(\beta))=0$.\par
The $A$-characteristic of $K$ is defined to be $\text{char}_A(K) = \ker(\gamma)$. For $a\not\in\text{char}_A(K)$, $\phi_a(x)$ is a separable polynomial, and we have the isomorphism of $A$-modules
\begin{align*}
\phi[a]\cong \prod_{i=1}^r A\slash aA.
\end{align*}
In this case, the $A$-module $\phi[a]$ is naturally equipped with an action of the absolute Galois group $G_K =\Gal(K^{\text{sep}}\slash K)$, since $\phi[a]$ consists of the roots of the separable polynomial $\phi_a(x)$. The action commutes with the action of $A$ on $\phi[a]$ since $\phi_a$ has coefficients in $K$. Thus, we obtain a representation
\begin{align*}
\rho_{\phi,a}: G_K \rightarrow \Aut_A(\phi[a])\cong \text{GL}_r(A\slash aA).
\end{align*}
Composing with the determinant map $\det:\text{GL}_r(A\slash aA)\rightarrow (A\slash aA)^{\times} \cong \text{GL}_1(A\slash aA)$, we get a homomorphism
\begin{align*}
\pi_{\phi,a}: G_K\rightarrow \text{GL}_1(A\slash aA).
\end{align*}
It is a natural question to ask whether $\pi_{\phi,a}$ arises from the action of $G_K$ on the torsion points of a Drinfeld module $\psi$ of rank 1 . \par
The answer to this question is yes, and in fact
\begin{align*}
\psi_t = \gamma(t) + (-1)^{r-1} g_{r}(t) \tau.
\end{align*}
The fact that $\rho_{\psi,a} = \det \circ \rho_{\phi,a}$ follows from the existence of the \textit{Weil-pairing} for Drinfeld modules, which is a map
\begin{align*}
W_a:\prod_{i=1}^r \phi[a] \rightarrow \psi[a]
\end{align*} 
satisfying the following properties:\\[.5 em]
(1) The map $W_a$ is $A$-multilinear, i.e. it is $A$-linear in each component.\\[.5 em]
(2) It is alternating: if $\beta_l=\beta_h$ for some $l\neq h$, then $W_a(\beta_1,\ldots,\beta_r) = 0$.\\[.5 em]
(3) Assuming $a\not\in \text{char}_A(K)$, it is surjective and nondegenerate:
\begin{align*}
\text{if }W_a(\beta_1,\ldots,\beta_r)=0 \quad\text{for all }\beta_1,\ldots,\beta_{i-1},\beta_{i+1},\ldots,\beta_r\in \phi[a], \quad\text{then } \beta_i=0.
\end{align*}
(4) It is Galois invariant:
\begin{align*}
\sigma W_a(\beta_1,\ldots,\beta_r) = W_a(\sigma\beta_1,\ldots,\sigma\beta_r) \quad \text{for all } \sigma\in \Gal(\overline{K}\slash K).
\end{align*}		
(5) It satisfies the following compatibility condition for $a,b\not \in \text{char}_A(K)$ and $\beta_1,\ldots,\beta_r\in \phi[ab]$:
\begin{align*}
\psi_{b}\big(W_{ab}(\beta_1,\ldots,\beta_r)\big)= W_a(\phi_b(\beta_1),\ldots,\phi_b(\beta_r)).
\end{align*}
The existence of $W_a$ was proven by van der Heiden in~\cite{vanderheiden} using the machinery of Anderson motives.
\begin{remark}
	Let $\pp=\text{char}_A(K)$. Let $\mathfrak{l}\neq \pp$ be a nonzero prime ideal in $A$, and denote the $\mathfrak{l}$-adic Tate module of $\phi$ by $T_\mathfrak{l}(\phi) = \lim \limits_{\longleftarrow} \phi[l]$. Property (5) says that the Weil-pairing extends to a map
	\begin{align*}
	W_{\mathfrak{l}^\infty}:\prod_{i=1}^r T_\mathfrak{l}(\phi)\rightarrow T_\mathfrak{l}(\psi),
	\end{align*}
	and it can be shown that this map inherits the properties (1) - (4). The Galois group $G_K$ acts on $T_\mathfrak{l}(\phi)$, so again we obtain a representation
	\begin{align*}
	\widehat{\rho}_{\phi,\mathfrak{l}}:G_k\rightarrow \Aut_{A_\mathfrak{l}}(T_\mathfrak{l}(\phi))\cong \text{GL}_r(A_\mathfrak{l}).
	\end{align*} As in the case of $W_a$, the existence of the map $W_\infty$ implies that 
	\begin{align*}
	\det \circ\widehat{\rho}_{\phi,\mathfrak{l}} = \widehat{\rho}_{\psi,\mathfrak{l} }.
	\end{align*}
\end{remark}
\begin{remark}
Let $\phi$ be a Drinfeld module of rank $r$ over the finite field $k\cong\F_{q^n}$, and let $\pp$ be the monic generator of $\text{char}_A(k)$. Denote the Frobenius automorphism of $\overline{k}$ by $\pi=\tau^n$. Let $P(x)$ be the characteristic polynomial of $\pi$ acting on $T_\mathfrak{l}(\phi)$. One uses the existence of the map $W_\infty$ mentioned above to deduce an explicit formula for $P(0)$ given by
\begin{align*}
P(0) = (-1)^{rn -r-n}\cdot \text{Nr}_{k\slash \F_q}(g_r)^{-1} \cdot \pp^{n\slash d}.
\end{align*}
See~\cite{Gek}.
\end{remark}
The goal of this paper is to give an alternative proof of the existence of $W_a$ with the required properties using explicit formulas. In~\cite{vanderheiden}, $W_a$ is shown to exist, but is only given explicitly for the case where $r=2$. The advantage of our approach is that we arrive at an explicit formula for $W_a$, in arbitrary rank $r$, as a polynomial in $K[x_1,\ldots,x_r]$; this formula matches van der Heiden's in the case where $r=2$. Moreover, because of the nature of our approach, the proof that we obtain is completely elementary, in the sense that no theory of Anderson motives is used. \par
The inspiration for our approach and some of the arguments come from Papikian's notes~\cite{Pap}, where van der Heiden's formula for $W_a$ in $r=2$ is used to derive the properties of $W_a$ in this case. Additionally, a formula for $r\geq 2$ and $a=T^n$ is given in~\cite{Pap}. As we explain in \S3, our formula given for $W_a$ recovers both the formula for $a=T^n$ given in~\cite{Pap} and the formula for $r=2$ given in~\cite{vanderheiden}. Now we explain in more detail our approach and the structure of our paper. \par
The compatibility condition (5) and the $\F_q$-multilinearity of $W_a$ together imply that for $c\in \F_q$, $W_{ca}(\beta_1,\ldots,\beta_r) = c^{r-1}W_a(\beta_1,\ldots,\beta_r)$, so it is enough to give a formula for $W_a$ when $a$ is monic of degree $\geq 1$. In \S 3, we describe an interesting class of symmetric polynomials $\{f_a\}\subset \overline{\F}_q[T_1,\ldots, T_r]$ with $f_a$ given for each monic $a\in \overline{\F}_q[T]$ with degree $\geq 1$. The polynomials are defined explicitly in \S3, and we will see that if we factor $a$ over $\overline{\F}_q$ as $a=(T-\alpha_1)\cdots(T-\alpha_n)$, the polynomials also satisfy the recursive formula for $1\leq l\leq r$,
\begin{align*}
f_a(T_1,\ldots,T_r) = \bigg(\prod_{\substack{j=1\\j\neq l}}^r(T_j-\alpha_1)\bigg) f_{a\slash(T-\alpha_1)}(T_1,\ldots,T_r) &+\bigg(\prod_{i=2}^n(T_l-\alpha_i)\bigg) f_a(T_1,\ldots,\widehat{T}_l,\ldots,T_r)
\end{align*}
whenever $r\geq 2$ and $n:=\deg(a)\geq 2$. If $n=1$ or $r=1$, then $f_a$ is the constant 1 polynomial. If $a\in A$, then $f_a\in \F_q[T_1,\ldots, T_r]$.\par
In \S4, we will show that a formula for $W_a$ incorporates the coefficients of $f_a$ and the Moore determinant
\begin{align*}
M(x_1,\ldots,x_r)=\det\begin{pmatrix}
x_1 & x_1^q &\cdots & x_1^{q^{r-1}}\\
x_2 & x_2^q &\cdots & x_2^{q^{r-1}}\\
\vdots & \vdots & \cdots &\vdots\\
x_r & x_r^q & \cdots & x_r^{q^{r-1}}
\end{pmatrix}.
\end{align*}
Namely, denote
\begin{align*}
f_a(T_1,\ldots, T_r) = \sum_{\substack{\mathbf{i}=(i_1,\ldots,i_r)\\0\leq i_1,\ldots, i_r\leq n-1}} a_\mathbf{i}T_1^{i_1}\cdots T_r^{i_r}, \qquad a_\mathbf{i}\in \F_q.
\end{align*} 
Then, a formula for the Weil-pairing $W_a$ with the desired properties is the following polynomial.
\begin{align*}
W_a(x_1,\ldots,x_r) = \sum_{\substack{\mathbf{i}=(i_1,\ldots,i_r)\\0\leq i_1,\ldots, i_r\leq n-1}} a_\mathbf{i}M(\phi_{T^{i_1}}(x_1),\ldots,\phi_{T^{i_r}}(x_r)).
\end{align*}
The fact that $W_a$ satisfies the five properties listed above will follow from the various properties and congruence conditions developed in \S3 for the polynomials $\{f_a\}$.\par
In the future, we are planning on considering other applications of the polynomials $\{f_a\}$ to the theory of Drinfeld modules. For example, they may relate to the theory of zeta-values of Drinfeld modules.

\section{Notation and preliminaries}
Given an $\F_q$-module $X$, we set $X^*=\overline{\F}_q\otimes_{\F_q} X$, so that $X^*$ is obtained from $X$ through extension of scalars to $\overline{\F}_q$. Given an $\F_q$-module homomorphism (resp. $\F_q$-algebra homomorphism) $\Phi:X\rightarrow Y$, there is an associated $\overline{\F}_q$-module homomorphism (resp. $\overline{\F}_q$-algebra homomorphism), $\Phi:X^*\rightarrow Y^*$, given for elementary tensors by
\begin{align*}
\Phi(c\otimes x)= c\otimes \Phi(x)
\end{align*} and then extended linearly. Since $A^*\cong \overline{\F}_q[T]$, we can extend any Drinfeld module $\phi$ in this way to a homomorphism of $\overline{\F}_q$-algebras $\phi:\overline{\F}_q[T]\rightarrow (K\{\tau\})^*$. \par
Let $K[x_1,\ldots,x_r]$ be the ring of polynomials with coefficients in $K$ and indeterminates $x_1,\ldots, x_r$. Let $K\langle x_1,\ldots,x_r\rangle$ be the subset of these polynomials with $q$-power exponents, which are the polynomials of the form
\begin{align*}
p(x_1,\ldots,x_r) = \sum_{i=1}^n \alpha_i x_1^{q^{j_{i,1}}}\cdots x_r^{q^{j_{i,r}}}, \qquad n\in \N,\; \alpha_i\in K,\;j_{i,\l}\in \N\cup \{0\}.
\end{align*}
For each polynomial $p\in K[x_1,\ldots,x_r]$, there is a function $p:(\overline{K})^r \rightarrow \overline{K}$ given by
\begin{align*}
(\beta_1,\ldots,\beta_r)\mapsto p(\beta_1,\ldots,\beta_r).
\end{align*}  \par
For any $p\in K\langle x_1,\ldots, x_r\rangle$, we can verify that $p:(\overline{K})^r\rightarrow \overline{K}$ is $\F_q$-multilinear, i.e. for $\beta_1,\ldots,\beta_{i-1},\beta_{i+1},\beta_r\in \overline{K}$, the polynomial $g(x)=p(\beta_1,\ldots,\beta_{i-1},x,\beta_{i+1},\ldots,\beta_r)$ is $\F_q$-linear. (This follows because $g(x)$ has $q$-power exponents, hence it is $\F_q$-linear.) \par
By the universal property of tensor products, given any multilinear polynomial $p\in K\langle x_1,\ldots,x_r\rangle$, there is a unique $\F_q$-linear map $\widehat{p}:\bigotimes_{i=1}^r (\overline{K}) \rightarrow \overline{K}$ such that $p(x_1,\ldots,x_r)=\widehat{p}(x_1\otimes\cdots\otimes x_r)$, which can be extended to an $\overline{\F}_q$-linear map
\begin{align*}
\widehat{p}:\bigg(\bigotimes_{i=1}^r (\overline{K})\bigg)^*\rightarrow \overline{K}^*
\end{align*}
by the discussion above.\par
Let $M\in K\langle x_1,\ldots,x_r\rangle$ be the Moore determinant, given by
\begin{align*}
M(x_1,\ldots,x_r)=\det\begin{pmatrix}
x_1 & x_1^q &\cdots & x_1^{q^{r-1}}\\
x_2 & x_2^q &\cdots & x_2^{q^{r-1}}\\
\vdots & \vdots & \cdots &\vdots\\
x_r & x_r^q & \cdots & x_r^{q^{r-1}}
\end{pmatrix}.
\end{align*}
The induced map $\widehat{M}:\big(\bigotimes_{i=1}^r (\overline{K})\big)^*\rightarrow \overline{K}^*$ will play a key role in \S 4.\par
From the homomorphism $\phi: \overline{\F}_q[T]\rightarrow (K\{\tau\})^*$, we can build the homomorphism of $\overline{\F}_q$-algebras
\begin{align*}
\bigg(\bigotimes_{i=1}^r \phi\bigg) : \bigg(\bigotimes_{i=1}^r \overline{\F}_q[T]\bigg) \rightarrow \bigg(\bigotimes_{i=1}^r (K\{\tau\})^*\bigg),\\
a_1\otimes \cdots \otimes a_r \mapsto \phi(a_1)\otimes\cdots \otimes \phi(a_r),
\end{align*}
where the tensors are over $\overline{\F}_q$. These tensors make sense since $\overline{\F}_q$ belongs to the center of $\overline{\F}_q[T]$ and $K\{\tau\}^*$.
We also have isomorphisms of $\overline{\F}_q$-algebras
\begin{align*}
\bigotimes_{i=1}^r \overline{\F}_q[T] \cong \overline{\F}_q[T_1,\ldots,T_r] \qquad \text{and} \qquad \bigotimes_{i=1}^r (K\{\tau\})^*\cong \bigg(\bigotimes_{i=1}^r (K\{\tau\})\bigg)^*,
\end{align*}
where now $\bigotimes_{i=1}^r (K\{\tau\})$ is the tensor product over $\F_q$. Hence, the map $\bigotimes_{i=1}^r \phi$ induces a homomorphism of $\overline{\F}_q$-algebras
\begin{align*}
\overline{\F}_q[T_1,\ldots,T_r] \rightarrow \bigg(\bigotimes_{i=1}^r (K\{\tau\})\bigg)^*,
\end{align*}
which extends to the chain of $\overline{\F}_q$-algebra homomorphisms
\begin{align*}
\overline{\F}_q[T_1,\ldots,T_r] \rightarrow \bigg(\bigotimes_{i=1}^r (K\{\tau\})\bigg)^*\hookrightarrow \End_{\overline{\F}_q}\bigg(\bigg(\bigotimes_{i=1}^r {\overline{K}}\bigg)^*\bigg),
\end{align*} 
through the action of $\overline{\F}_q$ on the $\overline{\F}_q$-component of $\big(\bigotimes_{i=1}^r {\overline{K}}\big)^*$ by scaling and the action of $K\{\tau\}$ on $\overline{K}$ given for monomials by $\alpha\tau \cdot \beta = \alpha\beta^q$.\par
We conclude that there is a natural action of $\overline{\F}_q[T_1,\ldots,T_r]$ on $\big(\bigotimes_{i=1}^r {\overline{K}}\big)^*$ induced by $\phi$ given for monomials on elementary tensors by
\begin{align*}
(\alpha T_1^{j_1}\cdots T_r^{j_r}) \cdot 1 \otimes \beta_1 \otimes \cdots \otimes \beta_r = \alpha\otimes \phi_{T^{j_1}}(\beta_1) \otimes \cdots \otimes \phi_{T^{j_r}}(\beta_r).
\end{align*}
This action furthermore restricts to an action of $\F_q[T_1,\ldots,T_r]$ on $\bigotimes_{i=1}^r \overline{K}$.
\section{A special class of polynomials in $\F_q[T_1,\ldots,T_r]$} 
For $r,n\in \N$ and monic $a\in \overline{\F}_q[T]$ with $\deg(a)=n$, we factor $a$ into linear factors, say $a(T)=(T-\alpha_1)\cdots(T-\alpha_n)$ for $\alpha_1,\ldots,\alpha_n\in \overline{\F}_q$ (not necessarily distinct). Define
\begin{align*}
f_a(T_1,\ldots,T_r)&:= \sum_{1=i_0\leq\cdots \leq i_r =n} \prod_{j=1}^r \frac{a(T_j)}{(T_j-\alpha_{i_{j-1}})\cdots (T_j-\alpha_{i_j})}\\
&=\sum_{1=i_0\leq\cdots \leq i_r =n} \prod_{j=1}^r \prod_{\substack{i=1\\i\not\in[i_{j-1},i_j]}}^n (T_j-\alpha_i),
\end{align*}
so that $f_a\in \overline{\F}_q[T_1,\ldots,T_r]$. Note that for fixed $1\leq j\leq r$, we are allowing $i_{j-1}=i_j$ in the summation indexing for $f_a$: if this is the case, the $j$-th factor of the corresponding product will be
\begin{align*}
\frac{a(T_j)}{T_j-\alpha_{i_{j-1}}} = \frac{a(T_j)}{T_j-\alpha_{i_{j}}}.
\end{align*} Examples follow.
\begin{example}\label{fainitialcases}
	If either $r=1$ or $n=1$, then $f_a$ is the constant $1$ polynomial. To see this, first suppose that $r=1$. Then,
	\begin{align*}
	f_a(T_1) = \frac{a(T_1)}{(T_1 - \alpha_1)\cdots (T_1-\alpha_{n}) }=1.
	\end{align*}
	Now suppose that $n=1$. Then, $a(T) = T-\alpha_1$, and
	\begin{align*}
	f_a(T_1,\ldots,T_r) = \prod_{j=1}^r \bigg(\frac{T_j - \alpha_1}{T_j-\alpha_1}\bigg)=1.
	\end{align*}
\end{example}
\begin{example}
	We compute $f_a$ in the case that $a(T)=T^n$. We start by noting that $\alpha_1=\cdots=\alpha_n=0$. Therefore,
	\begin{align*}
	f_a(T_1,\ldots,T_r) = \sum_{1=i_0\leq\cdots \leq i_r =n} \prod_{j=1}^r T_j^{n-(i_j-i_{j-1}+1)}.
	\end{align*}
	Thus, each monomial of $f_a$ is of the form $T_1^{j_1}\cdots T_r^{j_r}$, where $0\leq j_1, \ldots, j_r\leq n-1$ and
	\begin{align*}
	j_1+\cdots+j_r &= r(n-1)+(-i_r+i_{r-1})+(-i_{r-1}+i_{r-2})+\cdots+(-i_1+i_0)\\
	&=r(n-1) -n+1 =(r-1)(n-1).
	\end{align*}
	Furthermore, any monomial satisfying these conditions appears as a monomial of $f_a$, so we can write
	\begin{align*}
	f_a(T_1,\ldots,T_r) = \sum_{\substack{0\leq j_1,\ldots,j_r\leq n-1\\j_1+\cdots+j_r=(r-1)(n-1)}} T_1^{j_1}\cdots T_r^{j_r}.
	\end{align*}
\end{example}
\begin{example}
	We compute $f_a$ when $n=2$. Say $a=(T-\alpha_1)(T-\alpha_2)=T^2+a_1T+a_0$. Then,
	\begin{align*}
	f_a(T_1,\ldots,T_r)&=\sum_{i=1}^r (T_1-\alpha_2)(T_2-\alpha_2)\cdots(T_{i-1}-\alpha_2)(T_{i+1}-\alpha_1)\cdots(T_r-\alpha_1)\\
	&=\sum_{s=1}^r \bigg((-1)^{s-1}(\alpha_1^{s-1}+\alpha_1^{s-2}\alpha_2+\cdots+\alpha_1\alpha_2^{s-2}+\alpha_2^{s-1})\sum_{\substack{0\leq j_1,\ldots,j_r\leq 1\\j_1+\cdots+j_r=r-s}}T_1^{j_1}\cdots T_r^{j_r}\bigg).
	\end{align*}
	In this example, there is a recursive formula relating the coefficients of $f_a$. Note that the coefficients of $a$ are given by $a_1=-\alpha_1-\alpha_2$ and $a_0=\alpha_1\alpha_2$. Let $g(1)=1$, let $g(2)=a_1$, and define $g(s+2)=a_1g(s+1)-a_0g(s)$ for each $s\in \N$. By induction, it is easily seen that
	\begin{align*}
g(s)=(-1)^{s-1}(\alpha_1^{s-1}+\alpha_1^{s-2}\alpha_2+\cdots+\alpha_1\alpha_2^{s-2}+\alpha_2^{s-1}),
	\end{align*} so that
	\begin{align*}
	f_a(T_1,\ldots,T_r)=\sum_{s=1}^r \bigg(g(s)\sum_{\substack{0\leq j_1,\ldots,j_r\leq 1\\j_1+\cdots+j_r=r-s}}T_1^{j_1}\cdots T_r^{j_r}\bigg).
	\end{align*}
\end{example}
\begin{example}
	We compute $f_a$ when $n=3$ and $r=3$. Write $a(T)=(T-\alpha_1)(T-\alpha_2)(T-\alpha_3)=T^3+a_2T^2+a_1T+a_0$. First, we list the summation indices for $f_a$: There are 6 tuples $(i_0,i_1,i_2,i_3)$ with $0=i_0\leq i_1\leq i_2\leq i_3=3$, namely
	\begin{align*}
	(1,1,1,3),\, (1,1,2,3), \,(1,1,3,3),\, (1,2,2,3),\,(1,2,3,3), \,\text{and}\, (1,3,3,3).
	\end{align*}
	Thus,
	\begin{align*}
	f_a(T_1,T_2,T_3)=& (T_1-\alpha_2)(T_1-\alpha_3)(T_2-\alpha_2)(T_2-\alpha_3)\\
	+&(T_1-\alpha_2)(T_1-\alpha_3)(T_2-\alpha_3)(T_3-\alpha_1)\\
	+&(T_1-\alpha_2)(T_1-\alpha_3)(T_3-\alpha_1)(T_3-\alpha_2)\\
	+&(T_1-\alpha_3)(T_2-\alpha_1)(T_2-\alpha_3)(T_3-\alpha_1)\\
	+&(T_1-\alpha_3)(T_2-\alpha_1)(T_3-\alpha_1)(T_3-\alpha_2)\\
	+&(T_2-\alpha_1)(T_2-\alpha_2)(T_3-\alpha_1)(T_3-\alpha_2).
	\end{align*}
	If we distribute each of these six terms and use the fact that $a_2=-(\alpha_1+\alpha_2+\alpha_3)$, $a_1=\alpha_1\alpha_2 + \alpha_1\alpha_3 + \alpha_2\alpha_3$, and $a_0=-\alpha_1\alpha_2\alpha_3$, then we will arrive at
	\begin{align*}
	f_a(T_1,T_2,T_3) =& T_1^2T_2^2 + T_1^2T_2T_3 +T_1^2T_3^2+ T_1T_2^2T_3 + T_1T_2T_3^2 + T_2^2T_3^2 \\
	+&a_2(T_1^2T_2 +T_1^2T_3+ T_1T_2^2+T_1T_3^2+ T_2^2T_3+T_2T_3^2)\\
	+&2a_2T_1T_2T_3\\
	+&a_1(T_1^2+T_2^2+T_3^2)\\
	+&a_2^2(T_1T_2+T_1T_3+T_2T_3)\\
	+&(a_1a_2-a_0)(T_1+T_2+T_3)\\
	+&(a_1^2-a_0a_2).
	\end{align*}
\end{example}
For these examples, note that $f_a$ is symmetric in its variables $T_1, \ldots, T_r$ and that if the coefficients of $a$ lie in $\F_q$, then the coefficients of $f_a$ do as well. The following lemmas and theorems will show that these properties carry over to the general case, and in them we will prove some other notable properties of the polynomials $f_a$.  At the end of the section, we will compute $f_a$ in the case where $r=2$.
\begin{lemma}\label{congruence}
	Let $I$ be the ideal of $\overline{\F}_q[T_1,\ldots,T_r]$ generated by $\{a(T_j):1\leq j\leq r\}$. Let $l,h\in \{1,\ldots,r\}$. We have the following congruence:
	\begin{align*}
	T_l f_a(T_1,\ldots,T_r)\equiv T_h f_a(T_1,\ldots,T_r) \quad(\text{\emph{mod}}\; I).
	\end{align*}
	\begin{proof}
		It suffices to prove that the statement holds for $l>1$ and $h=l-1$. Notice that
		\begin{align*}
		T_lf_a(T_1,\ldots,T_r) =& \sum_{1=i_0\leq\cdots \leq i_r =n} \frac{a(T_l)[(T_l-\alpha_{i_{l-1}})+\alpha_{i_{l-1}}]}{(T_l-\alpha_{i_{l-1}})\cdots (T_l-\alpha_{i_l})} \prod_{\substack{j=1\\j\neq l}}^r \frac{a(T_j)}{(T_j-\alpha_{i_{j-1}})\cdots (T_j-\alpha_{i_j})}\\
		=&\sum_{1=i_0\leq\cdots \leq i_r =n} \frac{a(T_l)(T_l-\alpha_{i_{l-1}})}{(T_l-\alpha_{i_{l-1}})\cdots (T_l-\alpha_{i_l})} \prod_{\substack{j=1\\j\neq l}}^r \frac{a(T_j)}{(T_j-\alpha_{i_{j-1}})\cdots (T_j-\alpha_{i_j})}\\
		&\hspace{.7 cm}+ \sum_{1=i_0\leq\cdots \leq i_r =n} \frac{a(T_{l-1})\alpha_{i_{l-1}}}{(T_{l-1}-\alpha_{i_{l-2}})\cdots (T_{l-1}-\alpha_{i_{l-1}})} \prod_{\substack{j=1\\j\neq l-1}}^r \frac{a(T_j)}{(T_j-\alpha_{i_{j-1}})\cdots (T_j-\alpha_{i_j})}.
		\end{align*}
		The first sum belongs to $I$ when $i_{l-1}=i_l$, so we can change the summation index (rename $i_{l-1}+1$ to $i_{l-1}$) to get
		\begin{align*}
		\equiv&\sum_{\substack{1=i_0\leq\cdots \leq i_r =n\\i_{l-1}\neq i_{l-2}}} \frac{a(T_{l-1})(T_{l-1}-\alpha_{i_{l-1}})}{(T_{l-1}-\alpha_{i_{l-2}})\cdots (T_{l-1}-\alpha_{i_{l-1}})} \prod_{\substack{j=1\\j\neq l-1}}^r \frac{a(T_j)}{(T_j-\alpha_{i_{j-1}})\cdots (T_j-\alpha_{i_j})}\\
		&\hspace{.7 cm}+ \sum_{1=i_0\leq\cdots \leq i_r =n} \frac{a(T_{l-1})\alpha_{i_{l-1}}}{(T_{l-1}-\alpha_{i_{l-2}})\cdots (T_{l-1}-\alpha_{i_{l-1}})} \prod_{\substack{j=1\\j\neq l-1}}^r \frac{a(T_j)}{(T_j-\alpha_{i_{j-1}})\cdots (T_j-\alpha_{i_j})} \quad(\text{mod} \;I).
		\end{align*}
		Up to addition by an element of $I$ we can also remove the condition $i_{l-1}\neq i_{l-2}$ from the first sum. Recombining the two sums, we arrive at
		\begin{align*}
		T_lf_a(T_1,\ldots,T_r)\equiv T_{l-1}\sum_{1=i_0\leq\cdots \leq i_r =n}  \prod_{j=1}^r \frac{a(T_j)}{(T_j-\alpha_{i_{j-1}})\cdots (T_j-\alpha_{i_j})}\quad(\text{mod}\; I),
		\end{align*}
		which completes the proof.
	\end{proof}
\end{lemma}
\begin{corollary}\label{congruence2}
	Suppose $f_{a\slash (T-\alpha_1)}$ is defined using the same ordering of roots as $f_a$, so that
	\begin{align*}
	f_{a\slash (T-\alpha_1)}=\sum_{2=i_0\leq\cdots \leq i_r =n} \prod_{j=1}^r \frac{a(T_j)\slash (T_j-\alpha_1)}{(T_j-\alpha_{i_{j-1}})\cdots (T_j-\alpha_{i_j})}.
	\end{align*}Then, for $l\in \{1,\ldots,r\}$, we have
	\begin{align*}
	(T_l-\alpha_1)f_a(T_1,\ldots,T_r) \equiv \bigg(\prod_{j=1}^r(T_j-\alpha_1)\bigg) f_{a\slash(T-\alpha_1)}(T_1,\ldots,T_r) \quad (\text{\emph{mod}} \; I).
	\end{align*}
\end{corollary}
\begin{proof}
	By Lemma \ref{congruence},
	\begin{align*}
	T_lf_a(T_1,\ldots,T_r) &\equiv T_1f_a(T_1,\ldots,T_r) \quad (\text{mod}\; I).
	\end{align*}
	Subtracting $\alpha_1f_a(T_1,\ldots,T_r)$ on both sides of this congruence, we arrive at
	\begin{align*}
	(T_l-\alpha_1)f_a(T_1,\ldots,T_r) &\equiv (T_1-\alpha_1)f_a(T_1,\ldots,T_r) \quad (\text{mod}\; I).
	\end{align*}
	Whenever $0= i_0 = i_1\leq i_2 \leq \cdots\leq i_r =n $, we have
	\begin{align*}
	(T_1-\alpha_1)\prod_{j=1}^r \frac{a(T_j)}{(T_j-\alpha_{i_{j-1}})\cdots (T_j-\alpha_{i_j})} \in I.
	\end{align*}
	We conclude that
	\begin{align*}
	(T_l-\alpha_1)f_a(T_1,\ldots,T_r)&\equiv (T_1-\alpha_1)\sum_{\substack{1=i_0\leq\cdots \leq i_r =n\\i_0\neq i_1}} \prod_{j=1}^r \frac{a(T_j)}{(T_j-\alpha_{i_{j-1}})\cdots (T_j-\alpha_{i_j})} \quad (\text{mod}\; I)\\
	&\equiv \sum_{2=i_0\leq\cdots \leq i_r =n} \prod_{j=1}^r \frac{a(T_j)}{(T_j-\alpha_{i_{j-1}})\cdots (T_j-\alpha_{i_j})} \quad (\text{mod}\; I)\\
	&=\bigg(\prod_{j=1}^r(T_j-\alpha_1)\bigg) f_{a\slash(T-\alpha_1)}(T_1,\ldots,T_r),	\end{align*}
	completing the proof.
\end{proof}
\begin{theorem}\label{fatheorem}
	The polynomial $f_a\in \overline{\F}_q[T_1,\ldots,T_r]$ satisfies the following properties.\\
	(1) For $r,n\geq 2$ and $1\leq l\leq r$,
	\begin{align*}
	f_a(T_1,\ldots,T_r) = \bigg(\prod_{\substack{j=1\\j\neq l}}^r(T_j-\alpha_1)\bigg) f_{a\slash(T-\alpha_1)}(T_1,\ldots,T_r) &+\bigg(\prod_{i=2}^n(T_l-\alpha_i)\bigg) f_a(T_1,\ldots,\widehat{T}_l,\ldots,T_r).
	\end{align*}
	(2) The polynomial $f_a(T_1,\ldots,T_r)$ is symmetric in its variables $T_1,\ldots, T_r$.\\
	(3) Let $S_n$ denote the symmetric group on $n$ elements. For $\sigma\in S_n$, denote
	\begin{align*}
	f_a^\sigma(T_1,\ldots,T_r) := \sum_{1=i_0\leq\cdots \leq i_r =n} \prod_{j=1}^r \frac{a(T_j)}{(T_j-\alpha_{\sigma(i_{j-1})})\cdots (T_j-\alpha_{\sigma(i_j)})}.
	\end{align*}
	Then, the equality $f_a = f_a^\sigma$ holds. In other words, $f_a$ does not depend on the order in which we labeled the roots $\alpha_1,\ldots,\alpha_n$.\\
	(4) If $a\in A$, then all coefficients of the polynomial $f_a$ lie in $\F_q$.
\end{theorem}
\begin{proof}
	(1) We proved in Corollary \ref{congruence2} that for all $l\in \{1,\ldots,r\}$,
	\begin{align*}
	(T_l-\alpha_1)f_a(T_1,\ldots,T_r) \equiv \bigg(\prod_{j=1}^r(T_j-\alpha_1)\bigg) f_{a\slash(T-\alpha_1)}(T_1,\ldots,T_r) \quad (\text{mod} \; I),
	\end{align*}
	where $I$ is the ideal of $\overline{\F}_q[T_1,\ldots,T_r]$ generated by $\{a(T_j): 1\leq j \leq r\}$. Any nonzero polynomial in $I$ has degree at least $n$ in at least one of its variables, and the polynomials $f_a$ and $f_{a/(T-\alpha_1)}$ have degree in all variables of $<n$ and $<n-1$ respectively. We can therefore compare the degrees of the polynomials in the above congruence to conclude that
	\begin{align*}
	(T_l-\alpha_1)f_a(T_1,\ldots,T_r) =& \bigg(\prod_{j=1}^r(T_j-\alpha_1)\bigg) f_{a\slash(T-\alpha_1)}(T_1,\ldots,T_r) \\&+ a(T_l) \sum_{\substack{1=i_0\leq\cdots \leq i_r =n\\i_{l-1}=i_l}} \prod_{\substack{j=1\\j\neq l}}^r \frac{a(T_j)}{(T_j-\alpha_{i_{j-1}})\cdots (T_j-\alpha_{i_j})}\\
	=&\bigg(\prod_{j=1}^r(T_j-\alpha_1)\bigg) f_{a\slash(T-\alpha_1)}(T_1,\ldots,T_r) \\&+a(T_l)f_a(T_1,\ldots,\widehat{T}_l,\ldots,T_r).
	\end{align*}
	Since $\overline{\F}_q[T_1,\ldots,T_r]$ is an integral domain, we can cancel a factor of $(T_l-\alpha_1)$ from both sides of the equation to arrive at
	\begin{align*}
	f_a(T_1,\ldots,T_r) = \bigg(\prod_{\substack{j=1\\j\neq l}}^r(T_j-\alpha_1)\bigg) f_{a\slash(T-\alpha_1)}(T_1,\ldots,T_r) &+\bigg(\prod_{i=2}^n(T_l-\alpha_i)\bigg) f_a(T_1,\ldots,\widehat{T}_l,\ldots,T_r),
	\end{align*}
	which completes the proof of (1).\\[1 em]
	(2) We will now show that statement (1) implies that $f_a(T_1,\ldots,T_r)$ is symmetric in $T_1,\ldots T_r$ for $a$ of degree $n\in \N$. We proceed by induction on the pair $(r,n)$. From Example \ref{fainitialcases}, we know that the statement is true when either $r=1$ or $n=1$. For the inductive hypothesis, let $r,n\in \N$ and assume we know that $f_a(T_1,\ldots,T_{r-1})$ is symmetric in $T_1,\ldots, T_{r-1}$ and that $f_b(T_1,\ldots,T_r)$ is symmetric in $T_1,\ldots,T_r$ whenever $b$ is a monic divisor of $a$ with degree $n-1$, where $f_b$ is defined using the same ordering of roots as $f_a$. Let $\sigma\in S_r$ be a permutation on $r$ elements. Since (1) holds for both $l$ and $\sigma(l)$, we can employ the inductive hypothesis to arrive at
	\begin{align*}
	f_a(T_{\sigma(1)},\ldots, T_{\sigma(r)}) =& \bigg(\prod_{\substack{j=1\\j\neq l}}^r(T_{\sigma(j)}-\alpha_1)\bigg) f_{a\slash(T-\alpha_1)}(T_{\sigma(1)},\ldots,T_{\sigma(r)}) \\&+\bigg(\prod_{i=2}^n(T_{\sigma(l)}-\alpha_i)\bigg) f_a(T_{\sigma(1)},\ldots,\widehat{T}_{\sigma(l)},\ldots,T_{\sigma(r)})\\
	=&\bigg(\prod_{\substack{j=1\\j\neq \sigma(l)}}^r(T_j-\alpha_1)\bigg) f_{a\slash(T-\alpha_1)}(T_1,\ldots,T_r) \\&+\bigg(\prod_{i=2}^n(T_{\sigma(l)}-\alpha_i)\bigg) f_a(T_1,\ldots,\widehat{T}_{\sigma(l)},\ldots,T_r)\\
	=&f_a(T_1,\ldots,T_r).
	\end{align*}
	Therefore, $f_a(T_1,\ldots, T_r)$ is symmetric in $T_1,\ldots, T_r$. By induction, we are done.\\[1 em]
	(3) We proceed again by induction on pairs $(r,n).$ If $r=1$ or $n=1$ the statement is true. For the induction hypothesis, let $ r\geq 1$ and $a$ have degree $n\geq 1$ and assume that in rank $r-1$, $f_a^\sigma = f_a$ is true for all $\sigma\in S_n$, and that in rank $r$, $f_b^{\sigma} =f_b$ is true for all $\sigma\in S_{n-1}$ whenever $b$ is a monic divisor of $a$ of degree $n-1$ and $f_b$ and $f_{b}^\sigma$ are defined using the same ordering of roots as $f_a$. \par
	First, if $\sigma\in S_n$ is such that $\sigma(1)=1$ then $\sigma$ restricts to a permutation of the roots of the polynomial $a\slash (T-\alpha_1)$. Applying the formula in statement (1) for $l=1$, 
	\begin{align*}
	f_a^{\sigma}(T_1,\ldots, T_r) =& \bigg(\prod_{j=2}^r(T_j-\alpha_{\sigma(1)})\bigg) f_{a\slash(T-\alpha_{1})}^\sigma(T_1,\ldots,T_r) \\&+\bigg(\prod_{i=2}^n(T_1-\alpha_{\sigma(i)})\bigg) f_a^\sigma(T_2,\ldots,T_r).
	\end{align*}
	By the induction hypothesis,
	\begin{align*}
	f_{a\slash(T-\alpha_{1})}^\sigma(T_1,\ldots,T_r) = f_{a\slash(T-\alpha_{1})}(T_1,\ldots,T_r) \quad\text{and}\quad f_a^\sigma(T_2,\ldots,T_r) = f_a(T_2,\ldots,T_r).
	\end{align*} 
	Since $\sigma(1)=1$ and $\sigma$ permutes the elements of $\{2,\ldots, n\}$, we arrive at
	\begin{align*}
	f_a^{\sigma}(T_1,\ldots, T_r) =& \bigg(\prod_{j=2}^r(T_j-\alpha_{1})\bigg) f_{a\slash(T-\alpha_{1})}(T_1,\ldots,T_r) +\bigg(\prod_{i=2}^n(T_1-\alpha_{i})\bigg) f_a(T_2,\ldots,T_r)\\
	=&f_a(T_1,\ldots,T_r),
	\end{align*}	
	which we've now proved whenever $\sigma(1)=1$. Now suppose $\rho\in S_n$ is the permutation with $\rho(j)= (n-j+1)$ for all $1\leq j\leq n$, i.e. $\rho$ is the permutation that reverses the order of the roots of $a$. Then,
	\begin{align*}
	f_a^\rho(T_1,\ldots,T_r) &=\sum_{1=i_0\leq\cdots \leq i_r =n} \prod_{j=1}^r \frac{a(T_j)}{(T_j-\alpha_{\rho(i_{j-1})})\cdots (T_j-\alpha_{\rho(i_j)})}\\
	&=\sum_{1=i_0\leq\cdots \leq i_r =n} \prod_{j=1}^r \frac{a(T_{n-j+1})}{(T_{n-j+1}-\alpha_{i_{j-1}})\cdots (T_{n-j+1}-\alpha_{i_j})}\\
	&=f_a(T_r,\ldots,T_1).
	\end{align*}
	Since $f_a(T_1,\ldots,T_r)$ is symmetric in $T_1,\ldots,T_r$ by (3), it follows that $f_a^\rho(T_1,\ldots,T_r)=f_a(T_1,\ldots,T_r)$.\par
	We've now proven that $f_a^\sigma(T_1,\ldots,T_r)=f_a(T_1,\ldots,T_r)$ when either $\sigma(1)=1$ or $\sigma=\rho$. The set of $\sigma$ which fix $f_a$ is closed under multiplication since $f_a^{\sigma_1\sigma_2} = (f_a^{\sigma_2})^{\sigma_1}$. To complete the proof, take note that the subgroup of $S_n$ generated by $\rho$ and all $\sigma$ with $\sigma(1)=1$ includes all transpositions and therefore must be all of $S_n$.\\[1 em]
	(4) Let $\sigma\in \Gal(\overline{\F}_q\slash \F_q)$. Then, $\sigma$ acts on $\overline{\F}_q[T_1,\ldots,T_r]$ as an automorphism of rings by acting on the coefficients of a given polynomial. If the coefficients of $a$ belong to $\F_q$, $\sigma$ permutes the roots of $a$. Thus, we can apply statement (4) in order to arrive at
	\begin{align*}
	\sigma(f_a(T_1,\ldots,T_r)) = f_a^\sigma(T_1,\ldots,T_r) =f_a(T_1,\ldots,T_r).
	\end{align*}
	We conclude that $f_a\in \F_q[T_1,\ldots,T_r]$.
\end{proof}\par
Since $f_a$ does not depend on the order in which we labeled the roots $\alpha_1,\ldots,\alpha_n$, it is now clear that we can strengthen the claim of Corollary \ref{congruence2}, leading us to the following corollary.
\begin{corollary}\label{congruence updated}
	For $l\in \{1,\ldots,r\}$ and $\alpha$ any root of $a$, the following congruence holds 
	\begin{align*}
	(T_l-\alpha)f_a(T_1,\ldots,T_r) \equiv \bigg(\prod_{j=1}^r(T_j-\alpha)\bigg) f_{a\slash(T-\alpha)}(T_1,\ldots,T_r) \quad (\text{\emph{mod}} \; I).
	\end{align*}
\end{corollary}
\begin{example}\label{farank2}
	We compute $f_a$ when $r=2$. We'll show that
	\begin{align*}
	f_a(T_1,T_2)&= \sum_{i=1}^n a_i\sum_{j=1}^i T_1^{j-1}T_2^{i-j}. 
	\end{align*}
	Here, $a_i$ is the $i$-th coefficient of the polynomial $a=a_0+a_1T+\cdots+a_nT^n$. If $n=1$, then $f_a=1=g_a$. If $n\geq 2$, then the recursive formula proved in Theorem \ref{fatheorem} states that for $r=2$ and $l=1$,
	\begin{align*}
	 f_a(T_1,T_2) &= (T_2-\alpha_1) f_{a\slash(T-\alpha_1)}(T_1,T_2) +\bigg(\prod_{i=2}^n(T_1-\alpha_i)\bigg) f_a(T_2)\\
	 &=(T_2-\alpha_1) f_{a\slash(T-\alpha_1)}(T_1,T_2) +\bigg(\frac{a(T_1)}{(T_1-\alpha_1)}\bigg).
	\end{align*}
	We proceed by induction on $n$. For each monic divisor $b$ of $a$ of degree $n-1$, write $b=\sum_{i=0}^{n-1} b_iT^i$ and assume we've already shown that
	\begin{align*}
	f_b(T_1,T_2) = \sum_{i=1}^{n-1} b_i\sum_{j=1}^i T_1^{j-1}T_2^{i-j}. 
	\end{align*}
	In particular, this holds when $b=a\slash (T-\alpha_1)=\sum_{i=0}^{n-1} b_iT^i$, so we arrive at the following calculation.
	\begin{align*}
	f_a(T_1,T_2) &=(T_2-\alpha_1) \sum_{i=1}^{n-1} b_i\sum_{j=1}^i T_1^{j-1}T_2^{i-j} +b(T_1)\\
	&=\bigg((T_2-\alpha_1) \sum_{i=1}^{n-1} b_i\sum_{j=1}^i T_1^{j-1}T_2^{i-j}\bigg) + \bigg(\sum_{i=0}^{n-1} b_i T_1^i\bigg)\\
	&=\bigg(\sum_{i=1}^{n-1} b_i\sum_{j=1}^i T_1^{j-1}T_2^{i+1-j}\bigg) + \bigg(\sum_{i=0}^{n-1} b_i T_1^i \bigg)+ \bigg(\sum_{i=1}^{n-1} (-\alpha_1b_i)\sum_{j=1}^i T_1^{j-1}T_2^{i-j}\bigg)\\
	&=\bigg(\sum_{i=0}^{n-1} b_{i}\sum_{j=1}^{i+1} T_1^{j-1}T_2^{i+1-j}\bigg) + \bigg(\sum_{i=1}^{n-1} (-\alpha_1b_i)\sum_{j=1}^i T_1^{j-1}T_2^{i-j}\bigg) \\
	&= \bigg(\sum_{i=1}^{n} b_{i-1}\sum_{j=1}^{i} T_1^{j-1}T_2^{i-j}\bigg) + \bigg(\sum_{i=1}^{n-1} (-\alpha_1b_i)\sum_{j=1}^i T_1^{j-1}T_2^{i-j}\bigg).
	\end{align*}
	Recall that $b_{n-1}=1=a_n$. Since $(T-\alpha_1)b=a$, we must have $a_i=b_{i-1}-\alpha_1b_i$ when $1\leq i\leq n-1$. We conclude that
	\begin{align*}
	f_a(T_1,T_2)= \sum_{i=1}^{n} a_i \sum_{j=1}^iT_1^{j-1}T_2^{i-j}.
	\end{align*}
\end{example}
\section{The Main Result} Given $a\in A$, we now wish to write down a formula for the map $W_a:\prod_{i=1}^r \phi[a] \rightarrow \psi[a]$. Van der Heiden~\cite{vanderheiden} provided the following formula involving the Moore determinant $M$ when the rank $r$ of $\phi$ is 2 and $a$ written as $\sum_{i=0}^n a_iT^i$.
\begin{align*}
W_a(\beta_1,\beta_2)= \sum_{i=1}^n a_i \bigg(\sum_{j=1}^{i} M\big(\phi_{T^{j-1}}(\beta_1),\phi_{T^{i-j}}(\beta_2)\big)\bigg).
\end{align*}
We now write the same formula in terms of the linear map $\widehat{M}$ and the action of $\F_q[T_1,\ldots,T_r]$ on $\bigotimes_{i=1}^r {\overline{K}}$.
\begin{align*}
W_a(\beta_1,\beta_2)&= \widehat{M}\big(g_a \cdot ( \beta_1\otimes \beta_2)\big),\\ g_a(T_1,T_2)&= \sum_{i=1}^n a_i\sum_{j=1}^i T_1^{j-1}T_2^{i-j}.
\end{align*}
Notice that $g_a(T_1,T_2) = f_a(T_1,T_2)$ as in Example \ref{farank2}.\par
We generalize the formula to arbitrary rank $r$ using the polynomials $f_a$ developed in the previous section: Let $M$ be the Moore determinant in rank $r$. Let $a\in A$ be a monic polynomial such that $\deg(a)\geq 1$. Recall that in this case, $f_a\in \F_q[T_1,\ldots,T_r]$, so we can define the following.
\begin{align*}
W_a:(\overline{K})^r&\rightarrow \overline{K}\\
W_a(\beta_1,\ldots,\beta_r)&= \widehat{M}(f_a\cdot(\beta_1\otimes \cdots\otimes \beta_r)).
\end{align*}
In other words, denote
\begin{align*}
f_a(T_1,\ldots, T_r) = \sum_{\substack{\mathbf{i}=(i_1,\ldots,i_r)\\0\leq i_1,\ldots, i_r\leq n-1}} a_\mathbf{i}T_1^{i_1}\cdots T_r^{i_r}, \qquad a_\mathbf{i}\in \F_q.
\end{align*} 
Then, $W_a$ is the following polynomial, which belongs to $K\langle x_1,\ldots,x_r\rangle$.
\begin{align*}
W_a(x_1,\ldots,x_r) = \sum_{\substack{\mathbf{i}=(i_1,\ldots,i_r)\\0\leq i_1,\ldots, i_r\leq n-1}} a_\mathbf{i}M(\phi_{T^{i_1}}(x_1),\ldots,\phi_{T^{i_r}}(x_r)).
\end{align*}
In the following lemma, we give an upper bound for the degree of $W_a(x_1,\ldots,x_r)$  in each variable $x_1,\ldots, x_r$ and show how the formula in one rank lower, $W_a(x_1,\ldots,x_{r-1})$, can be recovered from $W_a(x_1,\ldots,x_r)$.
\begin{lemma} \label{leadingterm}
	For monic $a\in A$ of degree $n$, the polynomial $W_a(x_1,\ldots,x_r)$ has degree $\leq q^{rn-1}$ in each variable $x_1,\ldots, x_r$. Furthermore, we can write \begin{align*}
		W_a(x_1,\ldots,x_r) = c W_a(x_1,\ldots,x_{r-1}) x_r^{q^{rn-1}} + h_a(x_1,\ldots,x_r).
	\end{align*}
	Here, $c=g_r^{n-1}$ for $g_r$ the leading coefficient of $\phi_T$, and $h_a(x_1,\ldots,x_r)$ is a polynomial which depends on $a$ and which has degree $<q^{rn-1}$ in $x_r$.
\end{lemma}
\begin{proof}
	Recall that the polynomial $f_a(T_1,\ldots,T_r)$ has degree $\leq n-1$ in each of its variables, $\phi_{T^i}$ has degree $q^{ir}$ for any $i\geq 0$, and the rank $r$ Moore determinant has degree $q^{r-1}$ in each variable. From the formula above, 
	\begin{align*}
	W_a(x_1,\ldots,x_r) = \sum_{\substack{\mathbf{i}=(i_1,\ldots,i_r)\\0\leq i_1,\ldots, i_r\leq n-1}} a_\mathbf{i}M(\phi_{T^{i_1}}(x_1),\ldots,\phi_{T^{i_r}}(x_r)),
	\end{align*}
	we conclude that $W_a(x_1,\ldots,x_r)$ has degree in each variable at most $q^{(n-1)r} \cdot q^{r-1} = q^{rn-1}$. Next, we recall our recursive formula for the case where $l=r$,
	\begin{align*}
	f_a(T_1,\ldots,T_r) = \bigg(\prod_{j=1}^{r-1}(T_j-\alpha_1)\bigg) f_{a\slash(T-\alpha_1)}(T_1,\ldots,T_r) &+\bigg(\prod_{i=2}^n(T_r-\alpha_i)\bigg) f_a(T_1,\ldots,T_{r-1}).
	\end{align*}
	 Considering the polynomial $f_a(T_1,\ldots,T_r)$ as a polynomial in $T_r$, its leading term is $f_a(T_1,\ldots,T_{r-1}) T_{r}^{n-1}$. Considering $M(x_1,\ldots,x_r)$ as a polynomial in $x_r$ and expanding down the last column of the determinant, we see that its leading term is $x_r^{q^{r-1}}M(x_1,\ldots,x_{r-1})$, where here we are also using $M$ to denote the rank $r-1$ Moore determinant.  Considering $W_a(x_1,\ldots,x_r)$ as a polynomial in $x_r$, it follows that the coefficient of the $x_r^{q^{rn-1}}$ term is given by
	 \begin{align*}
	 \widehat{M}(cf_a(T_1,\ldots,T_{r-1}) \cdot (x_1\otimes\ldots\otimes x_{r-1})) = cW_a(x_1,\ldots,x_{r-1}),
	 \end{align*}
	 where $c=g_r^{n-1}$. This completes the proof.
\end{proof}

Recall that $\psi_T=\gamma(T) + (-1)^{r-1}g_r\tau$, where $g_r$ is the leading coefficient of $\phi_T$.
\begin{theorem} \label{compatibility}
	Let $a,b\in A$ be monic polynomials with degree $\geq 1$. Then, the following condition holds whenever $\beta_1,\ldots,\beta_r\in \phi[ab]$:
	\begin{align*}
	\psi_{b}\big(W_{ab}(\beta_1,\ldots,\beta_r)\big)= W_a(\phi_b(\beta_1),\ldots,\phi_b(\beta_r)).
	\end{align*}
\end{theorem}
\begin{proof}
	First, notice that we have the following equality of polynomials in $K[x_1,\ldots,x_r]$.
	\begin{align*}
	\psi_T(M(x_1,\ldots,x_r))&=\gamma(T)\det\begin{pmatrix}
	x_1 & x_1^q &\cdots & x_1^{q^{r-1}}\\
	x_2 & x_2^q &\cdots & x_2^{q^{r-1}}\\
	\vdots & \vdots & \cdots &\vdots\\
	x_r & x_r^q & \cdots & x_r^{q^{r-1}}
	\end{pmatrix} + (-1)^{r-1}g_r \det\begin{pmatrix}
	x_1^q & x_1^{q^2} &\cdots & x_1^{q^{r}}\\
	x_2^q & x_2^{q^2} &\cdots & x_2^{q^{r}}\\
	\vdots & \vdots & \cdots &\vdots\\
	x_r^q & x_r^{q^2} & \cdots & x_r^{q^{r}}
	\end{pmatrix}.
	\end{align*}
	Without changing the value of the expression, we can perform the following $r-1$ operations to the second determinant: Interchange the columns 1 and 2, then 1 and 3, and so on until we have interchanged the columns 1 and $r$, each time multiplying by a factor of $-1$. Now apply $K$-linearity in the first column of both determinants to arrive at
	\begin{align*}
	\psi_T(M(x_1,\ldots,x_r))&=\det\begin{pmatrix}
	\gamma(T)x_1 & x_1^q &\cdots & x_1^{q^{r-1}}\\
	\gamma(T)x_2 & x_2^q &\cdots & x_2^{q^{r-1}}\\
	\vdots & \vdots & \cdots &\vdots\\
	\gamma(T)x_r & x_r^q & \cdots & x_r^{q^{r-1}}
	\end{pmatrix} + \det\begin{pmatrix}
	g_r x_1^{q^r} & x_1^q &\cdots & x_1^{q^{r-1}}\\
	g_rx_2^{q^r} & x_2^q &\cdots & x_2^{q^{r-1}}\\
	\vdots & \vdots & \cdots &\vdots\\
	g_rx_r^{q^r} & x_r^q & \cdots & x_r^{q^{r-1}}
	\end{pmatrix}\\[.5 em]
	&= \det\begin{pmatrix}
	\gamma(T)x_1+g_r x_1^{q^r} & x_1^q &\cdots & x_1^{q^{r-1}}\\
\gamma(T)x_2+g_r x_2^{q^r} & x_2^q &\cdots & x_2^{q^{r-1}}\\
	\vdots & \vdots & \cdots &\vdots\\
	\gamma(T)x_r+g_r x_r^{q^r} & x_r^q & \cdots & x_r^{q^{r-1}}
	\end{pmatrix}
	\end{align*}
	We can now perform the following $r-1$ operations to this determinant without changing its value: For $2\leq i\leq r$, add $g_i$ times the $i$-th column to column 1. We conclude that
	\begin{align*}
	\psi_T(M(x_1,\ldots,x_r))&= \det\begin{pmatrix}
	\phi_T(x_1) &x_1^q&\cdots & x_1^{q^{r-1}}\\
	\phi_T(x_2) & x_2^q &\cdots & x_2^{q^{r-1}}\\
	\vdots & \vdots & \cdots &\vdots\\
	\phi_T(x_r) & x_r^q & \cdots & x_r^{q^{r-1}}
	\end{pmatrix}
	\end{align*}
	For $l\in \{1,\ldots,r\}$, consider the square submatrix of the following matrix achieved by removing its $l$-th row.
	\begin{align*}
	\begin{pmatrix}
	x_1^q & x_1^{q^2} &\cdots & x_1^{q^{r-1}}\\
	x_2^q & x_2^{q^2} &\cdots & x_2^{q^{r-1}}\\
	\vdots & \vdots & \cdots &\vdots\\
	x_r^q & x_r^{q^2} & \cdots & x_r^{q^{r-1}}
	\end{pmatrix}.
	\end{align*}
	Define $M_l\in K\langle x_1,\ldots,x_r\rangle$ by setting $M_l(x_1,\ldots,x_r)$ to be $x_l$ multiplied by the determinant of this submatrix. By expanding the Moore determinant down its first column, we see that
	\begin{align*}
	M(x_1,\ldots,x_r) = \sum_{l=1}^r (-1)^{l-1} M_l(x_1,\ldots,x_r).
	\end{align*}
	Similarly, letting $\beta_1,\ldots,\beta_r\in \overline{K}$ and expanding the determinant $\psi_T(M(\beta_1,\ldots,\beta_r))$ down its first column, we see that
	\begin{align*}
	\psi_T(M(\beta_1,\ldots,\beta_r)) = \sum_{l=1}^r (-1)^{l-1} \widehat{M}_l(T_l \cdot (\beta_1\otimes \cdots \otimes \beta_r)),
	\end{align*}
	where $\widehat{M}_l:\big(\bigotimes_{i=1}^r (\overline{K})\big)\rightarrow \overline{K}$ is the $\F_q$-linear map induced by $M_l$ and $\cdot$ is the action of $\F_q[T_1,\ldots,T_r]$ on $\bigotimes_{i=1}^r (\overline{K})$ described in \S 2. Also recall from \S 2 that $\psi$ extends to a map $A^*\rightarrow K\{\tau\}^*$, that $\widehat{M}$ and $\widehat{M_l}$ extend to $\overline{\F}_q$-linear maps $\big(\bigotimes_{i=1}^r (\overline{K})\big)^*\rightarrow \overline{K}^*$ and that $\cdot $ extends to an action of $\overline{\F}_q[T_1,\ldots,T_r]$ on $\big(\bigotimes_{i=1}^r (\overline{K})\big)^*$. For $\alpha\in \overline{\F}_q$, we have
	\begin{align*}
	\psi_{T-\alpha}\big(\widehat{M}(1\otimes \beta_1\otimes \cdots \otimes \beta_r)\big) =&\big(1\otimes \psi_T(M(\beta_1,\ldots,\beta_r))\big) -\big(\alpha \otimes M(\beta_1,\ldots,\beta_r)\big)\\
	=&\sum_{l=1}^r (-1)^{l-1} \widehat{M}_l(T_l \cdot (1\otimes \beta_1\otimes \cdots \otimes \beta_r))\\
	 & -\sum_{l=1}^r (-1)^{l-1} \widehat{M}_l(\alpha \otimes \beta_1\otimes \cdots \otimes \beta_r)\\
	 =& \sum_{l=1}^r (-1)^{l-1} \widehat{M}_l((T_l-\alpha) \cdot (1\otimes \beta_1\otimes \cdots \otimes \beta_r)).
	\end{align*}
	Assume $a\in \overline{\F}_q[T]$ has degree $\geq 1$. Let $\alpha\in \overline{\F}_q$ be a root of $a$. By $\overline{\F}_q$-linearity we have
	\begin{align}\label{applying psi eq}
	\psi_{T-\alpha}\big(\widehat{M}(f_a\cdot (1\otimes \beta_1\otimes \cdots \otimes \beta_r)) \big) = \sum_{l=1}^r (-1)^{l-1} \widehat{M}_l((T_l-\alpha)f_a \cdot (1\otimes \beta_1\otimes \cdots \otimes \beta_r)).
	\end{align}
	Recall from Corollary \ref{congruence updated} that for $I$ the ideal generated by $\{a(T_j):1\leq j\leq r\}$, and for each $l\in \{1,\ldots,r\}$, the following congruence holds.
	\begin{align} \label{congruence eq}
	(T_l-\alpha)f_a(T_1,\ldots,T_r) \equiv \bigg(\prod_{j=1}^r(T_j-\alpha)\bigg) f_{a\slash(T-\alpha)}(T_1,\ldots,T_r) \quad (\text{mod} \; I).
	\end{align}
	Suppose that $v\in \big(\bigotimes_{i=1}^r \overline{K}\big)^*$ is annihilated by $I$, that is, for all $g\in I$, $g\cdot v =0$. Then, combining equation (\ref{applying psi eq}) with congruence (\ref{congruence eq}) above,
	\begin{align*}
	\psi_{T-\alpha}\big(\widehat{M}(f_a\cdot v) \big) &= \sum_{l=1}^r (-1)^{l-1} \widehat{M}_l\bigg(\bigg(\prod_{j=1}^r(T_j-\alpha)\bigg) f_{a\slash(T-\alpha)} \cdot v  \bigg)\\
	&=\widehat{M}\bigg(f_{a\slash(T-\alpha)}\cdot\bigg(\bigg(\prod_{j=1}^r(T_j-\alpha)\bigg)  \cdot v\bigg)\bigg),
	\end{align*}
	where $\big(\prod_{j=1}^r(T_j-\alpha)\big)  \cdot v$ is annihilated by the ideal generated by $\{a(T_j)\slash(T_j-\alpha):1\leq j\leq r\}$. If $b$ is a monic divisor of $a$ with $\deg(b)\geq 1$, we can factor $b$ into linear factors and repeatedly apply the calculation above for each of these factors to conclude that
	\begin{align*}
	\psi_{b}\big(\widehat{M}(f_a\cdot v)\big) =\widehat{M}\bigg(f_{a\slash b}\bigg(\prod_{j=1}^r b(T_j)\bigg)  \cdot v\bigg).
	\end{align*}
	Finally, assume $a,b\in A$ are monic polynomials with degree $\geq 1$. Let $I$ be the ideal generated by $\{a(T_j)b(T_j):1\leq j\leq r\}$. Let $\beta_1,\ldots,\beta_r\in \phi[ab]$. Set $v=(1\otimes \beta_1\otimes \cdots\otimes \beta_r)\in \big(\bigotimes_{i=1}^r \overline{K}\big)^*$. Since $\beta_1,\ldots,\beta_r\in \phi[ab]$, $v$ is annihilated by $I$. Therefore,
	\begin{align*}
	\psi_{b}\big(W_{ab}(\beta_1,\ldots,\beta_r)\big)&=\psi_{b}\big(\widehat{M}(f_{ab}\cdot v)\big)\\
	&=\widehat{M}\bigg(f_{a}\bigg(\prod_{j=1}^r b(T_j)\bigg)  \cdot v\bigg)\\
	&=W_a(\phi_b(\beta_1),\ldots,\phi_b(\beta_r)).
	\end{align*}
	This completes the proof of Theorem \ref{compatibility}.
\end{proof}
	\begin{theorem} \label{other properties} The map $W_a:\prod_{i=1}^r \phi[a]\rightarrow \psi[a]$ has the following properties:\\[.5 em]
		(1) It is $A$-multilinear, i.e. it is $A$-linear in each component.\\[.5 em]
		(2) It is alternating: if $\beta_l=\beta_h$ for some $l\neq h$, then $W_a(\beta_1,\ldots,\beta_r) = 0$.\\[.5 em]
		(3) Assuming $a\not\in \ker(\gamma)$, it is surjective and nondegenerate:
		\begin{align*}
		\text{if }W_a(\beta_1,\ldots,\beta_r)=0 \quad\text{for all }\beta_1,\ldots,\beta_{i-1},\beta_{i+1},\ldots,\beta_r\in \phi[a], \quad\text{then } \beta_i=0.
		\end{align*}
		(4) It is Galois invariant:
		\begin{align*}
		\sigma W_a(\beta_1,\ldots,\beta_r) = W_a(\sigma\beta_1,\ldots,\sigma\beta_r) \quad \text{for all } \sigma\in \emph{Gal}(\overline{K}\slash K).
		\end{align*}		
	\end{theorem}
\begin{proof}
	(1) Let $i\in \{1,\ldots,r\}$. We will show the $A$-linearity in the $i$-th component. Clearly $W_a$ is $\F_q$-linear in each component, since $M$ is $\F_q$-linear in each component and $\tau: x\mapsto x^q$ is $\F_q$-linear. Notice that for $\beta_1,\ldots,\beta_r\in \phi[a]$,
	\begin{align*}
	W_a(\beta_1,\ldots,\beta_{i-1}, \phi_{T}(\beta_i), \beta_{i+1},\ldots,\beta_r)&=\widehat{M}(T_if_a\cdot (1\otimes\beta_1\otimes \cdots\otimes \beta_r))\\
	&=\sum_{l=1}^r (-1)^{l-1}\widehat{M}_l(T_if_a\cdot (1\otimes\beta_1\otimes \cdots\otimes \beta_r)).
	\end{align*}
	Recall from Lemma \ref{congruence} that $T_if_a \equiv T_l f_a$ modulo $I$ for all $l\in \{1,\ldots, r\}$. As in the previous proof, these congruences allow us to write
	\begin{align*}
	W_a(\beta_1,\ldots,\beta_{i-1}, \phi_{T}(\beta_i), \beta_{i+1},\ldots,\beta_r)&=\sum_{l=1}^r (-1)^{l-1}\widehat{M}_l(T_lf_a \cdot (1\otimes\beta_1\otimes \cdots\otimes \beta_r))\\
	&=\psi_T(\widehat{M}(f_a\cdot (1\otimes\beta_1\otimes \cdots\otimes \beta_r)))\\
	&=\psi_T(W_a(\beta_1,\ldots,\beta_r)).
	\end{align*}
	Applying $\phi_T$ to the $i$-th component $m$ times then yields
	\begin{align*}
W_a(\beta_1,\ldots,\beta_{i-1}, \phi_{T^m}(\beta_i), \beta_{i+1},\ldots,\beta_r)&=\psi_{T^m}(W_a(\beta_1,\ldots,\beta_r)),
	\end{align*}
	which holds for $m\geq 0$. Let $b\in \F_q[T]$, and write $b(T)=\sum_{j=0}^m b_jT^j$. Then, by $\F_q$-multilinearity of $W_a$, we have shown that
	\begin{align*}
	W_a(\beta_1,\ldots,\beta_{i-1}, \phi_{b}(\beta_i), \beta_{i+1},\ldots,\beta_r)
	&=\sum_{j=0}^m b_jW_a(\beta_1,\ldots,\beta_{i-1}, \phi_{T^j}(\beta_i), \beta_{i+1},\ldots,\beta_r)\\
	&=\sum_{j=0}^m b_j \psi_{T^j} (W_a(\beta_1,\ldots,\beta_r))\\
	&=\psi_b(W_a(\beta_1,\ldots,\beta_r)).
	\end{align*}
	We conclude that $W_a$ is $A$-multilinear.\\[.5 em]
	(2) To show $W_a$ is alternating, we will use the fact that the Moore determinant $M$ is alternating and the fact that $f_a$ is a symmetric polynomial. Write
	\begin{align*}
	f_a(T_1,\ldots, T_r) = \sum_{\substack{\mathbf{i}=(i_1,\ldots,i_r)\\0\leq i_1,\ldots, i_r\leq n-1}} a_\mathbf{i}T_1^{i_1}\cdots T_r^{i_r}, \qquad a_\mathbf{i}\in \F_q.
\end{align*}
Let $1\leq l,h\leq r$, $l\neq h$. Fix $0\leq i_1,\ldots,i_r\leq n-1$ and $\mathbf{i}=(i_1,\ldots,i_r)$ and let $\mathbf{i}'=(i_1',\ldots,i_r')$ be the tuple you get from swapping the $l$-th and $h$-th component of $\mathbf{i}$ so that $i_l'=i_h$ and $i_h'=i_l$ and $i_j'=i_j$ for $j\neq l,h$. The fact that $f_a$ is symmetric implies that we have equality of the corresponding coefficients of $f_a$, that is, $a_\mathbf{i}=a_{\mathbf{i}'}$. Thus, we can write
\begin{align*}
f_a(T_1,\ldots,T_r) =&\bigg( \sum_{\substack{\mathbf{i}=(i_1,\ldots,i_r)\\0\leq i_1,\ldots, i_r\leq n-1\\i_l=i_h}} a_\mathbf{i}T_1^{i_1}\cdots T_r^{i_r}\bigg)\\ &\hspace{.7 cm}+ \bigg(\sum_{\substack{\mathbf{i}=(i_1,\ldots,i_r)\\0\leq i_1,\ldots, i_r\leq n-1\\i_l<i_h}} a_\mathbf{i}(T_1^{i_1}\cdots T_r^{i_r}+T_1^{i_1'}\cdots  T_r^{i_r'})\bigg).
\end{align*}
	Let $\beta_1,\ldots,\beta_r\in \overline{K}$ be such that $\beta_l=\beta_h$. Then,
	\begin{align*}
	W_a(\beta_1,\ldots,\beta_r) =& \bigg( \sum_{\substack{\mathbf{i}=(i_1,\ldots,i_r)\\0\leq i_1,\ldots, i_r\leq n-1\\i_l=i_h}} a_\mathbf{i}M(\phi_{T^{i_1}}(\beta_1),\ldots,\phi_{T^{i_r}}(\beta_r))\bigg)\\ &\hspace{.7 cm}+ \bigg(\sum_{\substack{\mathbf{i}=(i_1,\ldots,i_r)\\0\leq i_1,\ldots, i_r\leq n-1\\i_l<i_h}} a_\mathbf{i}(M(\phi_{T^{i_1}}(\beta_1),\ldots,\phi_{T^{i_r}}(\beta_r))+M(\phi_{T^{i_1'}}(\beta_1),\ldots,\phi_{T^{i_r'}}(\beta_r)))\bigg).
	\end{align*} In each term of the first sum, the $l$-th and $h$-th components of $M(\phi_{T^{i_1}}(\beta_1),\ldots,\phi_{T^{i_r}}(\beta_r))$ are equal. In this case, since $M$ is alternating, $M(\phi_{T^{i_1}}(\beta_1),\ldots,\phi_{T^{i_r}}(\beta_r))=0$, and therefore the first sum is $0$. In each term of the second sum, $M(\phi_{T^{i_1'}}(\beta_1),\ldots,\phi_{T^{i_r'}}(\beta_r))$ is the expression that results from swapping the $l$-th and $h$-th components of the expression $M(\phi_{T^{i_1}}(\beta_1),\ldots,\phi_{T^{i_r}}(\beta_r))$. In this case, since $M$ is alternating, $M(\phi_{T^{i_1}}(\beta_1),\ldots,\phi_{T^{i_r}}(\beta_r))+M(\phi_{T^{i_1'}}(\beta_1),\ldots,\phi_{T^{i_r'}}(\beta_r))=0$, and therefore the second sum is 0. We conclude that $W_a(\beta_1,\ldots,\beta_r)=0$, so $W_a$ is alternating.\\[.5 em]
	(3) We proceed by induction on the rank $r$ in order to prove that $W_a(x_1,\ldots,x_r)$ is surjective and nondegenerate. If $r=1$, then $\phi = \psi$ and $W_a(x_1)$ is the identity map $\phi[a]\rightarrow \psi[a]$. In particular, it is nondegenerate and surjective. Now let $r$ be arbitrary and assume for induction that $W_a(x_1,\ldots,x_{r-1})$ is nondegenerate and surjective. Assume $a\not\in \ker(\gamma)$ has degree $n$. Then, the polynomial $\phi_a(x)$ is separable of degree $q^{rn}$, so its set of zeros, $\phi[a]\subset \overline{K}$, has $q^{rn}$ distinct elements. Fix $i$ with $1\leq i\leq r-1$ and assume $\beta_i\in \phi[a]$ is such that $W_a(\beta_1,\ldots,\beta_r)=0$ for all $\beta_1,\ldots,\beta_{i-1},\beta_{i+1},\ldots,\beta_r\in \phi[a]$. If $\beta_i\neq 0$, then by the inductive hypothesis, we can fix $\beta_1,\ldots,\beta_{i-1},\beta_{i+1},\ldots,\beta_{r-1}\in \phi[a]$ such that $W_a(\beta_1,\ldots,\beta_{r-1})$ is nonzero. Consider the polynomial in $\overline{K}[x]$ given by
	\begin{align*}
	g(x) =W_a(\beta_1,\ldots,\beta_{r-1},x).
	\end{align*}
	By Lemma \ref{leadingterm}, $g(x)$ has degree exactly $q^{rn-1}$. Since $\phi[a]$ has $q^{rn}$ distinct elements, there is some $\beta_r\in \phi[a]$ with $g(\beta_r)=W_a(\beta_1,\ldots,\beta_r)\neq 0$, which is a contradiction, so $\beta_i=0$. In the case where $i=r$, if $W_a(\beta_1,\ldots,\beta_r)=0$ for all $\beta_1,\ldots,\beta_{r-1}\in \phi[a]$, then since $W_a$ is alternating, $W_a(\beta_1,\ldots, \beta_{r-2},\beta_r,\beta_{r-1})=0$ for all $\beta_1,\ldots,\beta_{r-1}\in \phi[a]$. By the proof above, we see that $\beta_r=0$. This proves $W_a(\beta_1,\ldots,\beta_r)$ is nondegenerate. \par
	For surjectivity, assume for contradiction that the image of $W_a(x_1,\ldots,x_r)$ is not equal to $\psi[a]\cong A\slash aA$. Since $W_a$ is $A$-multilinear, the image must be isomorphic to a proper $A$-submodule of $A\slash aA $, i.e. $bA\slash aA $ for some $b\mid a$ with $\deg(b)\geq 1$. Then, the polynomial $b'=a\slash b$ is a polynomial with $\deg(b')<\deg(a)$ such that 
	\begin{align*}
	\psi_{b'}(W_a(\beta_1,\ldots,\beta_r)) = 0 \quad \text{for all}\quad \beta_1,\ldots,\beta_r\in \phi[a].
	\end{align*}
	By the $A$-multilinearity of $W_a$, this equation can be written as
	\begin{align*}
	W_a(\phi_{b'}(\beta_1),\beta_2,\ldots,\beta_r) = 0 \quad\text{for all}\quad \beta_1,\ldots,\beta_r\in \phi[a].
	\end{align*}
	By the nondegeneracy of $W_a$, this means $\phi_{b'}(\beta_1)= 0$ for all $\beta_1\in \phi[a].$ This is a contradiction, since $b'$ has strictly smaller degree than $a$ and $a$ is separable. We conclude by induction that $W_a(x_1,\ldots, x_r)$ is surjective and nondegenerate for all $r\geq 1$.\\[.5 em]
	(4) Since $W_a(x_1,\ldots,x_r)\in K[x_1,\ldots,x_r]$, it immediately follows that $W_a$ is Galois invariant.
\end{proof}
\bibliographystyle{plain}{}
\bibliography{weilpairingbibtex}

\begin{thebibliography}{1}

\bibitem{Pap}
Mihran Papikian. \textit{Drinfeld Modules}. Course notes, 2018.

\bibitem{Gek}
Ernst-Ulrich Gekeler.
\newblock On finite {D}rinfel'd modules.
\newblock {\em J. Algebra}, 141(1):187--203, 1991.

\bibitem{vanderheiden}
Gert-Jan van~der Heiden.
\newblock Weil pairing for {D}rinfeld modules.
\newblock {\em Monatsh. Math.}, 143(2):115--143, 2004.

\end{thebibliography}

\end{document}